\documentclass[12pt]{article}
\usepackage[pdftex]{color, graphicx}
\usepackage{setspace}
\usepackage[utf8]{inputenc}
\usepackage{amsmath}
\usepackage[pdftex,bookmarks,colorlinks]{hyperref}
\usepackage{latexsym}
\usepackage{amsfonts}
\usepackage{amssymb}
\usepackage{amsthm}
\usepackage{paralist}
\usepackage{mathrsfs}
\begin{document}\newtheorem{thm}{Theorem}
\newtheorem{cor}[thm]{Corollary}
\newtheorem{lem}{Lemma}
\theoremstyle{remark}\newtheorem{rem}{Remark}
\theoremstyle{definition}\newtheorem{defn}{Definition}
%%%%%%%%%%%%%%%%%%%%%%%%%
\title{A Sufficient Condition For An Operator To Map $uL^\infty$ to ${\rm{BMO}}_u$}
\author{Sakin Demir}
%\date{November 05, 2016}
\maketitle
\begin{abstract} Let $T$ be an operator and suppose that there exists a positive constant $C$  such that
$$\left(\int_I|Tf(x)|^q\, dx\right)^{1/q}\leq C\left(\int_I|f(x)|^q\, dx\right)^{1/q}$$
for every $q$ which is near enough to $1$ and for every interval $I$ in $\mathbb{R}$ and $f\in L^{\infty}(\mathbb{R})$. Then we show that $T$ maps $uL^{\infty}$ to ${\rm{BMO}}_u$.
\end{abstract}
\centerline{}
\centerline{}
\noindent

We say that a non-negative function $w\in L_{{\rm{loc}}}^1(\mathbb{R})$ is an $A_p$ weight for some $1<p<\infty$ if
\begin{align*}
\sup_I\left[\frac{1}{|I|}\int_Iw(x)\,dx\right]\left[\frac{1}{|I|}\int_I[w(x)]^{-1/(p-1)}\,dx\right]^{p-1}\leq C<\infty
\end{align*}
The supremum is taken over all intervals $I\subset \mathbb{R}^n$; $|I|$ denotes the measure of $I$.\\
\indent
$w$ is an $A_\infty$ weight if given an interval $I$ there exists $\delta >0$ and $\epsilon >0$ such that for any measurable set $E\subset I$,
\begin{align*}
|E|<\delta\cdot |I|\implies w(E)<(1-\epsilon )\cdot w(I).
\end{align*}
Here 
$$w(E)=\int_Ew(x)\,dx.$$
$w\in A_1$ if
$$I(w)=\frac{1}{|I|}\int_I w(x)\, dx\leq C\,{\rm{ess\, inf}}_I\,w$$
for all intervals $I$.\\
It is well known and can easily be seen that $w\in A_\infty$ implies $w\in A_p$  if $1<p<\infty$.\\
Before presenting our main result  let us first mention some of the well known properties of the functions in $A_p$.
%%%%%%%%%%%%%%%%%%%%%%%%%%%%%%%%%%%%%%%%%%%%%%%%%%%%%%%%%%
\begin{lem}[Reverse Hölder Inequality]\label{rhi} Let $w$ satisfy $A_p$, where $1\leq p<\infty$. Then the inequality
$$\left(\frac{1}{|I|}\int_I(w(x))^{1+\delta }\,dx\right)^{\frac{1}{1+\delta}}\leq C\left(\frac{1}{|I|}\int_Iw(x)\,dx\right)$$
holds for all intervals $I\subset \mathbb{R}$, with constants $C$, $\delta >0$ independent of $I$.
\end{lem}
\begin{proof} This lemma is a part of the Corollary 2.13 in J. Garcia-Cuerva and J. L. Rubio de Francia~\cite{gcrdf1} on page 403.
\end{proof}
%%%%%%%%%%%%%%%%%%
Note also that it is  well known and easy to see that $A_p\subset A_r$ for $p<r$.\\
%%%%%%%%%%%%%%%%%%%%%%%%%%%%%%%%%%%%%%%%%%%%%%%%
Let $|I|$ denote the Lebesgue measure of the interval $I$ and let
$$I(f)=\frac{1}{|I|}\int_If.$$
We say that $f\in {\rm{BMO}}_u$ if
$$\frac{1}{|I|}\int_I|f-I(f)|\leq CI(u)$$
for all intervals $I$ in $\mathbb{R}$.\\
The sharp maximal function is defined as
$$f^\sharp (x)=\sup\left\{\frac{1}{|I|}\int_I|f-I(f)|:x\in I, \: I\:\textrm{an interval}\right\}.$$
If $1<p<\infty$, then
$$\|f^\sharp\|_p\approx \|f\|_p,$$
provided $\|f\|_p$ is finite.\\
\indent
Let $T$ be an operator. We say that $T:uL^{\infty}\to{\rm{BMO}}_u$ provided
$$\frac{1}{|I|}\int_I|T(fu)-I(Tfu)|\leq CI(u)\|f\|_{\infty},$$
with $C$ independent of $I$, and we say that $T:L^p(u)\to L^p(u)$ if
$$\|Tf\|_{L^p(u)}\leq C\|f\|_{L^p(u)}.$$
%%%%%%%%%%%%%%%%%%%%%%%%%%%%%%%%%
We can now state and prove our main result:
\begin{thm}\label{sdinter}Let $T$ be an operator and suppose that there exists a positive constant $C$  such that
$$\left(\int_I|Tf(x)|^q\, dx\right)^{1/q}\leq C\left(\int_I|f(x)|^q\, dx\right)^{1/q}$$
for every $q$ which is near enough to $1$ and for every interval $I$ in $\mathbb{R}$ and $f\in L^{\infty}(\mathbb{R})$. Then $T$ maps $uL^{\infty}$ to ${\rm{BMO}}_u$ for all $u\in A_1$.
\end{thm}
\begin{proof} For this let $u\in A_1$ and $f\in L^\infty$. We need to show that given an interval $I$ in $\mathbb{R}$
$$\frac{1}{|I|}\int_I|T(fu)-I(Tfu)|\leq CI(u)\|f\|_{\infty},$$
with a positive constant $C$ independent of $I$.\\
Let $q>1$ be near enough to $1$ so that reverse Hölder inequality holds for $u$ with exponent $q$. Then we have
\begin{align*}
\frac{1}{|I|}\int_I|T(fu)-I(Tfu)|&\leq \frac{1}{|I|}\int_I|T(fu)|+\frac{1}{|I|}\int_I|I(Tfu)|\\
&\leq \frac{1}{|I|}\int_I|T(fu)|+\frac{1}{|I|}\int_I|(Tfu)|\\
&= \frac{2}{|I|}\int_I|T(fu)|\\
&\leq 2\left(\frac{1}{|I|}\int_I|T(fu)|^q\right)^{1/q}\\
&\leq 2C\left(\frac{1}{|I|}\int_I|f|^qu^q\right)^{1/q}\quad (\textrm{by the hypothesis})\\
&\leq 2C\|f\|_{\infty}\left(\frac{1}{|I|}\int_Iu^q\right)^{1/q}\\
&\leq C_1\|f\|_{\infty}\frac{1}{|I|}\int_Iu\quad (\textrm{by reverse Hölder inequality})\\
&=C_1\|f\|_{\infty}I(u),
\end{align*}
where $C_1$ is a positive constant independent of $I$. This shows that $T$ maps $uL^\infty$ to ${\rm{BMO}}_u$.
\end{proof}

\vspace{1cm}
\noindent
Sakin Demir\\
E-mail: sakin.demir@gmail.com\\
Address:\\
İbrahim Çeçen University\\
 Faculty of Education\\
04100 Ağrı, TURKEY.


\begin{thebibliography}{99}
\bibitem{gcrdf1} J.~Garcia-Cuerva and J. L.~Rubio de Francia, 
\emph{Weighted norm inequalities and related topics},
North-Holland, Mathematics Studies 116, 1985.

\end{thebibliography}
\end{document}